\documentclass[reqno]{amsart}
\usepackage[T1]{fontenc}
\usepackage{amsmath,amssymb,amsfonts}
\usepackage{esint, cancel}
\usepackage{amsthm}
\usepackage{soul}
\usepackage{mathrsfs}
\usepackage{geometry}
\usepackage{comment}
\usepackage{xcolor}
\usepackage{lipsum}  
\usepackage{verbatim}
\usepackage{enumerate}
\usepackage{mathtools}
\usepackage{graphicx}
\usepackage{color}
\mathtoolsset{showonlyrefs}
\usepackage[colorlinks=true,urlcolor=blue, citecolor=red,linkcolor=blue,
linktocpage,pdfpagelabels, bookmarksnumbered,bookmarksopen]{hyperref}
\usepackage[hyperpageref]{backref}
\usepackage{soul}
\usepackage{tikz}

\newtheorem{theorem}{Theorem}[section]
\newtheorem{lemma}[theorem]{Lemma}
\newtheorem{proposition}[theorem]{Proposition}

\theoremstyle{definition}

\newtheorem{remark}[theorem]{Remark}

\newtheorem{definition}[theorem]{Definition}

\newtheorem*{ack}{Acknowledgments}
\newcommand{\R}{\mathbb{R}}
\newcommand{\N}{\mathbb{N}}

\newcommand{\bari}{\operatorname{bar}}

\numberwithin{equation}{section}
\subjclass[2020]{Primary 52A38, 49K40, 52A05; Secondary 28A75, 49Q10.}
\keywords{Fractional perimeter, quantitative isoperimetric inequality, barycentric asymmetry.}
\begin{document}
\title[]{Barycentric stability of nonlocal perimeters: the convex case} 
\date{\today}
\author[Gambicchia]{Chiara Gambicchia}
\address[C. Gambicchia]{
	Scuola Normale Superiore
	\newline\indent
	piazza dei Cavalieri 7,
	56127 Pisa, Italy}
\email{chiara.gambicchia@sns.it}
\author[Merlino]{Enzo Maria Merlino}
\address[E. M. Merlino]{Dipartimento di Matematica
	\newline\indent
	Alma Mater Studiorum  Universit\`a di Bologna
	\newline\indent
	piazza di Porta San Donato 5,
	40126 Bologna, Italy}
\email{enzomaria.merlino2@unibo.it}
\author[Ruffini]{Berardo Ruffini}
\address[B. Ruffini]{Dipartimento di Matematica
	\newline\indent
	Alma Mater Studiorum  Universit\`a di Bologna
	\newline\indent
	piazza di Porta San Donato 5,
	40126 Bologna, Italy}
\email{berardo.ruffini@unibo.it}
\author[Talluri]{Matteo Talluri}
\address[M.\ Talluri]{Dipartimento di Matematica
	\newline\indent
	Alma Mater Studiorum  Universit\`a di Bologna
	\newline\indent
	piazza di Porta San Donato 5,
	40126 Bologna, Italy}
\email{matteo.talluri@unibo.it}

\begin{abstract}

In this work, we establish a sharp form of a nonlocal quantitative isoperimetric inequality involving the barycentric asymmetry for convex sets. This result can be seen as the nonlocal analogue of the one obtained by Fuglede in \cite{Fuglede1993}. A main tool in the proof is an estimate from below of the fractional perimeter by a negative power of the inradius for convex sets.

\end{abstract}

\maketitle

\section{Introduction}

Quantitative isoperimetric inequalities have recently attracted considerable interest. The fundamental question of this research line is simple to state. It is well known that, among all sets with a given volume, the ball uniquely minimizes the perimeter, one may ask whether a set that nearly minimizes the perimeter must itself be close, in some precise sense, to a ball. Therefore, the goal is to establish a quantitative relation linking the perimeter excess of a set to its geometric proximity to a ball.

To formalize this, we  recall the notion of the \emph{isoperimetric deficit} of a set $E \subseteq \mathbb{R}^n$, defined as
$$
\delta(E) := \frac{P(E) - P(B(m))}{P(B(m))}\,,
$$
where $P(\cdot)$ is the perimeter in the sense of De Giorgi and $B(m)$ denotes the ball centered at the origin with volume $m = |E|$. {We point out that it is, by definition, scale invariant.}
Next, we need an index that measures how far a set is from being a ball, which leads to the definition of \emph{asymmetry}. In literature, different notions of asymmetry have been introduced{, all of which are scale invariant by definition}.

{One possible} choice is the \emph{Hausdorff asymmetry}, defined by
$$
\lambda_H(E) := \inf \left\{ \frac{d\big(E,(x+B(m))\big)}{|E|^\frac1n} :\, x\in\R^n \right\},
$$
where $d$ denotes the Hausdorff distance; see, e.g.,~\cite[Subsection~3.2]{F}.  
Indeed, after some first contributions in the planar case~\cite{Ber, Bon}, in higher dimensions Fuglede~\cite{Fuglede1989} proved that for any convex set $E$, up to explicit multiplicative constants depending on the dimension, one can estimate the Hausdorff asymmetry by a suitable power of the deficit $\delta(E)$, with the correct order of magnitude as $\delta\to 0$.  

However, it is not difficult to recognize that the Hausdorff asymmetry is a too strong a notion when dealing with general sets of finite perimeter; see, for instance,~\cite[Section~4]{F} for an explicit counterexample.  
As first recognized by Hall in~\cite{Hall}, the Hausdorff asymmetry can be replaced by the \emph{Fraenkel asymmetry index}:
\begin{equation}\label{def_Fr_as}
\lambda(E) := \min \left\{ \frac{\big|E\triangle (x+B(m))\big|}{|E|} :\, x\in\R^n \right\},
\end{equation}
where we denote by "$\triangle$" the symmetric difference, that is, for any pair of sets $A$ and $B$ in $\R^n$, $A\triangle B := (A\setminus B) \cup (B\setminus A)$.  
With this notion of asymmetry, as first proved in~\cite{FMP}, the sharp quantitative isoperimetric inequality states that, for any set $E$ of finite perimeter,
\begin{equation}\label{ShQII}
\lambda(E) \leq C_F(n) \sqrt{\delta(E)},
\end{equation}
where $C_F(n)$ is a constant depending only on the dimension $n$.  
Inequality~\eqref{ShQII} has then been reproved with different techniques; see, for instance,~\cite{FiMP, CL}.

Another notion of asymmetry, which is the one we focus on in this article, is the so-called \emph{barycentric asymmetry}, defined by
\begin{equation}\label{def_bar_as}
\lambda_0(E) := \frac{\big|E\triangle (\bari(E)+B(m))\big|}{|E|},
\end{equation}
where $\bari(E)$ denotes the barycenter of $E$. So, while with the Fraenkel asymmetry the optimal ball is chosen to minimize the volume of the symmetric difference with the set, in the case of the barycentric asymmetry the ball is simply the one centered at the barycenter of the set.  
This is a strong and somewhat arbitrary choice; however, it is reasonable to expect that in most cases if a set $E$ is very close to a ball, then the center of this ball cannot be too far from the barycenter of $E$. Working with the barycentric asymmetry thus becomes very convenient, since it avoids the need for an optimization procedure.  
For instance, in numerical approximations, it is clearly computationally much more efficient to compute the barycenter and then evaluate the volume of the symmetric difference, rather than performing a minimization process.

The corresponding sharp quantitative inequality, proved by Fuglede in~\cite{Fuglede1993}, reads as
\begin{equation}\label{ShBar}
\lambda_0(E) \leq C_B(N) \sqrt{\delta(E)},
\end{equation}
whenever $E$ is a convex set of finite perimeter.  
Observe that, as in~\eqref{ShQII} and unlike the case of Hausdorff asymmetry, the sharp exponent in the deficit is again $\frac12$.

As with the Hausdorff asymmetry, also with the barycentric one the inequality~\eqref{ShBar} is not valid for general sets and the same counterexample can be used to prove it, see \cite[Section 1]{GP}. However, recently, the estimate~\eqref{ShBar} has been extended under weaker geometric assumptions.  
In~\cite{BCH}, the authors showed that there exists a universal constant $C_{BCH}$ such that, for every connected set $E\subset \mathbb{R}^2$, one has
$$
\lambda_0(E) \leq C_{BCH} \sqrt{\delta(E)}.
$$
Moreover, in~\cite{GP} it was proved that, for every $n\geq 2$ and every $D>0$, there exists a constant $C(n,D)$ such that, for any set $E\subseteq \mathbb{R}^n$ with $\operatorname{diam}(E) \leq D |E|^{1/n}$, it holds
$$
\lambda_0(E) \leq C(n,D) \sqrt{\delta(E)}.
$$

Motivated by these results, in this paper we focus on convex sets and aim to extend the barycentric quantitative isoperimetric inequality \eqref{ShBar} to the fractional setting.   The notion of nonlocal perimeter was introduced in~\cite{CRS}.  

 Given $s\in(0,1)$, the $s$--perimeter of a measurable set $E\subseteq\mathbb{R}^n$ is defined as
\begin{equation}\label{defPs}
P_s(E):=\int_E\int_{E^c}\frac{dx\,dy}{|x-y|^{n+s}}.
\end{equation}
The study of fractional perimeters is motivated by several applications.  
They naturally arise as nonlocal generalizations of the classical perimeter, interpolating between the Lebesgue measure and the De Giorgi's perimeter functional, see, e.g.,~\cite{PonSpe}.  
Moreover, fractional perimeters appear in models with long-range interactions, such as phase transitions, dislocation dynamics, and nonlocal diffusion processes, see~\cite{BV}.  

The isoperimetric property of balls for the nonlocal perimeter was established in~\cite{FS}: for any measurable set $E\subset\mathbb{R}^n$ with $|E| = {|B(m)|}$, one has
\begin{equation}\label{non--local_iso}
P_s({B(m)}) \leq P_s(E),
\end{equation}
with equality if and only if $E$ is a ball. The sharp quantitative version of the isoperimetric inequality~\eqref{non--local_iso} was later proved in~\cite{FFMMM}.  
For every $n\geq 2$ and $s_0\in (0,1)$, there exists a positive constant $C(n, s_0)$ such that 
\begin{equation}\label{isop-F2M3}
\lambda(E) \leq C(n,s_0)\sqrt{\delta_s(E)},
\end{equation}
whenever $s\in[s_0,1]$ and $0<|E|<\infty$, where
\begin{equation}\label{s-def}
\delta_s(E) := \frac{P_s(E) - P_s({B(m)})}{P_s({B(m)})}, \qquad \text{with } {|B(m)|} = |E|,
\end{equation}
is the $s$--isoperimetric deficit.

The aim of this note is to initiate the study of barycentric quantitative isoperimetric inequalities in the fractional setting, focusing on convex sets. More precisely, we establish a lower bound for the $s$--isoperimetric deficit in terms of the barycentric asymmetry, valid for any convex set in $\mathbb{R}^n$. 

Our main result is the following.
\begin{theorem}\label{thm:main}
For any $s\in(0,1)$ and any convex set $E\subseteq\mathbb{R}^n$ with finite measure and nonempty interior, there exists a constant $C$, 
	depending only on $n$ and $s$, such that
	\begin{equation}\label{our-quant-iso}
		\lambda_0(E)\leq C\,\sqrt{\delta_s(E)}.
	\end{equation}
\end{theorem}
Although in this note we focus on convex sets, the proof works as well for nearly spherical sets, see Remark \ref{rmk:ns}. This is a key point for extending the result to a broader class of sets. However, as in the local case, we point out that a barycentric isoperimetric inequality cannot hold without additional assumptions on the class of admissible sets (e.g., equi--boundedness). Indeed, consider the set $E$ as the union of a ball of radius slightly less than one and a second ball of very small radius $\varepsilon$ placed sufficiently far apart so that the barycenter lies outside the set $E$ and total volume equals that of the unit ball. In this configuration, we have $\lambda_0(E) = 2$, while $\delta_s(E) \approx \varepsilon^{n-s}$.

{We finally point out the recent paper \cite{CPP2025} where the authors show some barycentric quantitative  inequalities in capillarity problems, with techniques which may be adapted to the nonlocal case}

The paper is organized as follows. In Section 2  we will recall standard facts about fractional Sobolev spaces. In Section 3 we will prove a continuity result for the  {$s$--}isoperimetric deficit in terms of the barycentric asymmetry. Finally, in Section 4 we will prove our main result, that is, the lower bound of the $s$--isoperimetric deficit in terms of the barycentric asymmetry.

\section{Setting and main result}
In this section, we recall some basic facts about the fractional perimeter and the corresponding isoperimetric properties of balls.

From \eqref{defPs}, it is easy to see that the $s$--perimeter is translation and rotation invariant, and for any $\lambda>0$, there holds $P_s(\lambda E)=\lambda^{n-s}P_s(E)$.  
Moreover, note that for sets with finite $s$--perimeter it holds
$$P_s(E)=\iint_{\R^{2n}}\frac{\chi_E(x)\chi_{E^c}(y)}{|x-y|^{n+s}}\,dxdy=\frac12\iint_{\mathbb R^{2n}}\frac{\left|\chi_E(x)-\chi_E(y)\right|}{\left|x-y\right|^{n+s}}dx\,dy=\frac 1 2 [\chi_E]_{W^{s,1}(\R^n)},$$
where $[\chi_E]_{W^{s,1}(\R^n)}$ denotes the Gagliardo $W^{s,1}$--seminorm of the characteristic function of $E$ and $W^{s,1}(\mathbb R^n)$ is the fractional Sobolev space defined by
\begin{equation*}
W^{s,1}(\R^n):=\left\{u\in L^1(\R^n):\,\iint_{\R^{2n}}\frac{|u(x)-u(y)|}{|x-y|^{n+s}}\,dx\,dy<\infty\right\}.
\end{equation*}

Since for $s\in (0,1)$ the space $BV(\R^n)$ is embedded in $W^{s,1}(\R^n)$, see \cite[Proposition 2.1]{Lom} the $s$--perimeter of $E$ is finite if $E$ has
finite {De Giorgi} perimeter and finite measure. On the other hand, $P_s(E)$ can be finite even if the Hausdorff dimension of $\partial E$ is strictly greater than $n-1$, see, for instance, \cite[Theorem 1.1]{Lom}. In particular, since convex sets are of locally finite perimeter (and bounded convex sets are of finite perimeter), convex sets are of locally finite $s$--perimeter (and bounded convex sets are of finite $s$-perimeter).

Furthermore, the $s$--perimeter can be seen as a fractional interpolation between the De Giorgi's perimeter (recovered in the limit $s\to 1$) and the $n$--dimensional Lebesgue measure (corresponding to $s \to 0$). More precisely, it can be shown
\begin{equation}\label{convergenza}
\lim_{s\nearrow1}\,(1-s)P_s(E)= \omega_{n-1}\,P(E),
\end{equation} where $P(\cdot)$ is the De Giorgi perimeter and $\omega_{n-1}=\mathcal{H}^{n-1}(\partial B)$. The asymptotic result \eqref{convergenza} was first obtained by combining the seminal work by Bourgain, Brezis and Mironescu \cite[Theorem 3 and Remark 4]{BBM} with a result by D\'avila \cite{davila2002open}. On the other hand,  as a consequence of {a result by Maz'ya and Shaposhnikova, that is,} \cite[Theorem 3]{MR1940355}, we have, for any set $E$ of finite measure and finite $s$--perimeter
\begin{equation}\label{convergenza0}
\lim_{s\searrow 0}sP_s(E)=n\omega_n|E|,
\end{equation} where $\omega_n=|B|$.

In order to deal with the barycentric quantitative version of the isoperimetric inequality \eqref{non--local_iso}, we recall the notion of asymmetry that we will use in the following.
\begin{definition}\label{defbar}
    Given a set $E\subseteq\R^n$ with positive measure, we define the \emph{barycenter} of $E$ as \[
    \operatorname{bar}(E)=\fint_Ex\:dx,
    \]and the \emph{barycentric asymmetry} of $E$ as \[
    \lambda_0(E)=\frac{|E \triangle(\operatorname{bar}{(E)}+B(m))|}{|E|},
    \]
   {where, as above, $B(m)$ denotes the ball centered at the origin and with the same volume as $E$.}
\end{definition}
\noindent From the above definition, for every $E\subseteq\mathbb{R}^n$, we clearly have
$$\lambda(E)\leq\lambda_0(E)\leq2.$$
As in the local case treated in \cite{CL}, the starting point to prove \eqref{isop-F2M3} is a Fuglede-type result, for \emph{nearly spherical} sets, see \cite[Theorem 2.1]{FFMMM}.

\begin{definition}
    \label{def:nearlyspherical}An open and bounded set $E\subseteq\mathbb R^n$ with $|E|=|B|$ and barycenter at the origin is \emph{nearly spherical} if \[
    \partial E=\left\{\left(1+u(x)\right)x\:|\:x\in\partial B\right\}
    \]for some $u\in W^{1,\infty}(\partial B)$ with $\|u\|_{W^{1,\infty}(\partial B)}$ sufficiently small.
\end{definition}

\begin{theorem}\label{fuglede}
 There exist two constants $\varepsilon_0\in(0,1/2)$ and $c_0>0$, depending only on $n$, such that, if $E$ is a nearly spherical set with $\|u\|_{W^{1,\infty}(\Omega)}<\varepsilon_0$, then
\begin{equation}\label{fug0}
 P_s(E)-P_s(B)\geq c_0\,\Big([u]_{\frac{1+s}{2}}^2+s\,P_s(B)\,\|u\|_{L^2(\partial B)}^2\Big)\,,\qquad\text{ for all } s\in(0,1)\,,
 \end{equation}
where the Gagliardo seminorm $[u]_{\frac{1+s}{2}}$ is {defined as}
\[
[u]_{\frac{1+s}{2}}^2 = \iint_{\partial B\times\partial B}\frac{|u(x)-u(y)|^2}{|x-y|^{n+s}}\,d\mathcal{H}^{n-1}_x\,d\mathcal{H}^{n-1}_y.
\]
\end{theorem}
\section{Preliminary results}

The aim of this section is essentially the reduction to the {\it small--deficit regime}.
In other words, we prove that for convex sets, if the $s$--isoperimetric deficit is sufficiently small, then the barycentric asymmetry must be small as well.
The proof of this fact is divided into two steps.
First, we prove that the statement holds for uniformly bounded sets, and then we prove that convex sets with  small deficit are uniformly bounded. In what follows we will denote by $Q_l$ the cube of side $l.$
\begin{lemma}\label{lm:decifit picclo baricentro piccolo}
   {Let $l>0$ and $s\in(0,1)$} {be given}. Then, for every $\varepsilon>0$ there exists $\eta=\eta(n,s,l,\varepsilon)$ such that, for any set $E\subset Q_l$ with volume $|E|=\omega_n$ and barycenter at the origin, if $\delta_s(E)\leq \eta$ {holds}, then $\lambda_0(E)\leq \varepsilon$.
\end{lemma}
\begin{proof}
    Fix a positive $\varepsilon$ and assume by contradiction that such an $\eta$ does not exist. 
    Then there exists a sequence of sets $\{E_j\}$ such that:
    \begin{itemize}
        \item For every $j\in\N$, $E_j\subset Q_l$;
        \item For every $j\in\N$, $|E_j|=\omega_n$ and bar$(E_j)=0$;
        \item $\delta_s(E_j)\to 0$ as $j\to +\infty$;
        \item For every $j\in\N$, it holds $\lambda_0(E_j)>\varepsilon>0$.

    \end{itemize}
    Hence,  the $\chi_{E_j}$'s are uniformly bounded in $W^{1,s}(Q_l)$ with 
    \[ \sup_{j\in\mathbb N}\int_{Q_l}\int_{Q_l}\frac{|\chi_{E_j}(x)-\chi_{E_j}(y)|}{|x-y|^{n+s}}\,dx\,dy<+\infty\]
    and so, due to \cite[Theorem 7.1]{DPV}, and recalling that every $E_j$ is contained in $Q_l$, we can assume, up to a subsequence, that \[
    \chi_{E_j}\longrightarrow \chi_{E}
    \]
   strongly in $L^1$, as $j\to \infty$, for some set $E$ of finite $s$--perimeter. In particular, the limit $E$ will have volume \[
   |E|=\|\chi_E\|_{L^1}=\lim_{j\to+\infty}\|\chi_{E_j}\|_{L^1}=\omega_n,
   \]and, by the Dominated Convergence Theorem, the limit set $E$ will have barycenter equal to $0$.
  Moreover, its fractional perimeter will be 
   \[
   P_s(E)\leq \liminf_{j\to+\infty}P_s(E_j)=P_s(B),
   \]
   since $\delta_s(E_j)\to 0$ and the fractional perimeter is lower semicontinuous with respect to the $L^1$--convergence. 
   Hence, by \eqref{non--local_iso}, we conclude that $E=B$ and so $|E_j\triangle B|=\|\chi_E - \chi_{E_j}\|_{L^1}\to 0$, which contradicts the assumption that $\lambda_0(E_j)=|E_j\triangle B|/\omega_n>\varepsilon\,$ for any $j\in\N$.
\end{proof}

{Now, we recall some classical notions and results concerning support and
separation of convex sets. For further details, we refer to
\cite[Section~1.3]{Sch-book}.

Let $ E\subset \mathbb{R}^n$ be any set and let $H \subset \mathbb{R}^n$ be a hyperplane. Let us denote by $H^+$ and $H^-$ the two closed halfspaces bounded by $H$. We say that $H$ \emph{supports} $E$ at a point $x$ if $x \in E \cap H$ and either $E \subset H^+$ or $E \subset H^-$. Moreover, $H$ is called a \emph{support hyperplane} of $E$ if it supports $E$ at some point $x$, which is necessarily a boundary point of $E$.

Let $ H_{u,p}$ be the hyperplane passing thought $p$ and orthogonal to the direction $u\in \mathbb{S}^{n-1}$. If $ H_{u,p}$ supports $E$ at $p\in \partial E$ and
\[
E \subset H^+_{u,p}
:= \{ y \in \mathbb{R}^n : \langle y-p, u \rangle \geq 0 \},
\]
 then $H^+_{u,p}$ is called a \emph{supporting halfspace} of $E$ and $u$ is called an \emph{outer unit normal vector} of both $H_{u,p}$ and $H^+_{u,p}$.

For any closed and convex set $E \subset \mathbb{R}^n$ there exists a support hyperplane of $E$ through each boundary point of $E$ (see, e.g., \cite[Theorem~1.3.2]{Sch-book}).

\begin{remark}\label{rmk:bound-conv}
Let $E\subset\mathbb{R}^n$ be a convex set with nonempty interior and
$0<|E|<\infty$. Then $E$ must be bounded. Indeed, let
$p\in\operatorname{int}(E)$, then there exists $\varepsilon>0$ such that
$B(p,\varepsilon)\subset E$. Assume by contradiction that $E$ is unbounded.
Then there exists a sequence $(x_k)\subset E$ such that
$\|x_k-p\|\to\infty$. For each $k\in\mathbb{N}$, by convexity,
\[
\operatorname{conv}\big(B(p,\varepsilon)\cup\{x_k\}\big)\subset E,
\]
where $\operatorname{conv}\big(B(p,\varepsilon)\cup\{x_k\}\big)$ denotes the
convex hull of $B(p,\varepsilon)\cup\{x_k\}$, i.e.\ the intersection of all
convex subsets of $\mathbb{R}^n$ containing $B(p,\varepsilon)\cup\{x_k\}$.
Therefore,
\[
|E|\geq
\big|\operatorname{conv}\big(B(p,\varepsilon)\cup\{x_k\}\big)\big|
\geq \frac{1}{n}\,\omega_{n-1}\,\varepsilon^{\,n-1}\,\|x_k-p\|
\longrightarrow \infty,
\]
which contradicts the assumption $|E|<\infty$.
\end{remark}}

In the next result, we show that convex sets with small deficit satisfy the hypotheses of Lemma \ref{lm:decifit picclo baricentro piccolo}. {To the best of our knowledge, related results may already exist in the literature, possibly in different frameworks. However, since we could not locate a specific reference addressing this setting, we present a self-contained argument. As a byproduct, our approach highlights a relationship between the fractional perimeter, the volume, and the inradius for convex sets (see equation \eqref{ps-vol-inr} below), which might be of independent interest.}

\begin{proposition}\label{prop:deficit_piccolo_implica_cubo}
    There exists a constant $\widetilde{l}=\widetilde{l}(n,s)$ such that, for any convex set $E$ with nonempty interior, $|E|=\omega_n$ and $\delta_s(E)\leq 1$, there holds 
    \[
    {\rm diam}(E)\leq \widetilde{l}.
    \]
\end{proposition}
\begin{proof}
Since $E$ is convex and $|E|<+\infty$, for any $x\in E$ there exists $p\in\partial E$ such that $|x-p|=\operatorname{dist}(x,\partial E)=:d$.
{We claim that
\[
\langle x-p, y-p\rangle \ge 0 \qquad\text{for any } y\in E,
\]}
{ as shown in figure \ref{fig:acuteangle}.}
\begin{figure}
    \centering
    \begin{tikzpicture}
        \filldraw(0,-0.05) circle (0.5pt)node[above] {$p$};
        \draw(2.85,-0.7) arc[start angle=60, end angle=120, radius=5];
        \node at (2.9,-0.4) {$\partial E$};
        \filldraw(0,-1.05) circle (0.5pt)node[below] {$x$};
        \filldraw(2,-1.05) circle (0.5pt)node[right] {$y$};
        \draw[dashed](0,-0.05)--(0,-1.05);
        \draw[dashed](0,-0.05)--(2,-1.05);
    \end{tikzpicture}
    \caption{The angle $x\hat{p}y$ is acute for any $y\in E$.}
    \label{fig:acuteangle}
\end{figure}

{
Indeed, assume by contradiction that there exists $y\in E$ such that $\langle x-p,\; y-p\rangle<0$. Now consider the line through $p$ and $y$, i.e. $x_t:=p+t(p-y)$ with $t\in \R$. For $t>0$ sufficiently small, since $\langle x-p,\; y-p\rangle<0$, we have 
\[
|x-x_t|^2=|(x-p)-t(p-y)|^2 =|x-p|^2+2t\,\langle x-p,\; y-p\rangle+t^2|y-p|^2<|x-p|^2=d^2,
\]
which implies that $x_t\in B(x,d)\subset E$. Hence, we have $x_t\in int(E)$. But since $y\in E$ according to \cite[Theorem 6.1]{Rock-book}, we must have $p\in int(E)$ leading to a contradiction.

Thus, in particular, the hyperplane
\[
H:=\{y\in\mathbb{R}^n:\langle y-p,\; x-p\rangle=0\}
\]
is a supporting hyperplane to $E$ at $p$, and its {inward} unit normal vector is $u=\frac{x-p}{|x-p|}$. Consequently, up to a set of measure zero, $$H^-=H^-_{u,p}:=\{y\in\R^n\,|\,\langle y-p\,,\, u\rangle\leq 0\}\subseteq E^c.$$Then, we have 
\begin{equation}\label{perimetro dal basso col piano}
P_s(E)=\int _E\int_{E^c}\frac{dy\,dx}{|x-y|^{n+s}}\geq\int_E\int_{H^-}\frac{dy}{|x-y|^{n+s}}\ dx 
.
\end{equation}
Now, up to rotations and traslations, we can assume that $u=-e_n$ and $p= -\operatorname{dist}(x, \partial E)e_n$. Then, it holds \[
\int_{H^-}\frac{dy}{|x-y|^{n+s}}=\int_{\{y_n\leq -\operatorname{dist}(x,\partial E)\}}\frac{dy}{|y|^{n+s}}.
\]
By writing $y=(y',y_n)$ with $y'\in \R^{n-1}$ and $y_n\in \R$, we find \begin{align}
\int_{\{y_n\leq -\operatorname{dist}(x,\partial E)\}}\frac{dy}{|y|^{n+s}}&=\int^{-\operatorname{dist}(x,\partial E)}_{-\infty}\int_{\R^{n-1}}\frac{d y'}{(|y'|^2+y_n^2)^{\frac{n+s}{2}}}\ dy_n\\
&=\int^{-\operatorname{dist}(x,\partial E)}_{-\infty}\frac{1}{|y_n|^{n+s}}\int_{\R^{n-1}}\frac{d y'}{(|\frac{y'}{y_n}|^2+1)^{\frac{n+s}{2}}} \ dy_n\\\label{catena integrali}
&=\int_{\operatorname{dist}(x,\partial E)}^{+\infty}\frac{dy_n}{|y_n|^{1+s}}\int_{\R^{n-1}}\frac{d z'}{(|z'|^2+1)^{\frac{n+s}{2}}}\\
&=\frac{\alpha_{n,s}}{\operatorname{dist}(x,\partial E)^s}
\end{align}
for some constant $\alpha_{n,s}>0$, depending only on $n$ and $s$, where in \eqref{catena integrali} we used the change of variable $z'=\frac{y'}{y_n}$. By inserting this in \eqref{perimetro dal basso col piano}, we get \begin{equation}\label{perimetro distanza}
P_s(E)\geq \alpha_{n,s}\int_E\frac{dx}{\operatorname{dist}(x,\partial E)^s}.
\end{equation}
Let $r_E$ be the inradius of $E$. In light of Remark \ref{rmk:bound-conv} we have $r_E<\infty$ and for any $x\in E$, $\operatorname{dist}(x,\partial E)\leq r_E$. Hence by \eqref{perimetro distanza} we infer 
\begin{equation}\label{ps-vol-inr}
P_s(E)\geq \alpha_{n,s}\frac{|E|}{r_E^s}=\frac{\alpha_{n,s}\omega_n}{r_E^s}.
\end{equation}
Moreover, since $\delta_s(E)\leq 1$ we have $P_s(E)\leq 2P_s(B)$ which together with \eqref{ps-vol-inr} ensure
\begin{equation}\label{inradius non troppo piccolo}
r_E\geq \left(\frac{\omega_n\alpha_{n,s}}{2P_s{(B)}}\right)^\frac{1}{s}.
\end{equation}
Now we claim that there exists a constant $c_n>0$, depending only on $n$, such that 
\begin{equation}
\omega_n=|E|\geq c_n r_E^{n-1}\operatorname{diam }(E)\label{diametro inradius},
\end{equation}
Then \eqref{diametro inradius}, together with  \eqref{inradius non troppo piccolo}, gives
\[
\operatorname{diam}(E)\leq\frac{1}{c_n}\left(\frac{\omega_n\alpha_{n,s}}{2P_s{(B)}}\right)^\frac{1-n}{s},
\]
concluding the proof}
{, since the right--hand side depends only on $n$ and $s$.}

{Therefore, we have to prove only \eqref{diametro inradius}. By the definition of inradius, there exists $x_E\in E$ such that $B(x_E,r_E)\subseteq E$. Let $x_0\in \partial E$ and $x_1\in\partial E$ such that $|x_1-x_0|=\operatorname{diam}(E)$. Then, at least one between $x_0$ and $x_1$ has distance from $x_E$ at least $\frac{\operatorname{diam}(E)}{2}$. Indeed, if this is not the case, we would have 
\[
\operatorname{diam}(E)=|x_0-x_1|\leq|x_0-x_E|+|x_1-x_E|<\frac{\operatorname{diam}(E)}{2}+\frac{\operatorname{diam}(E)}{2}=\operatorname{diam}(E)
\]
which gives a contradiction. Thus, up to interchanging the roles of $x_0$ and $x_1$, we can assume $|x_0-x_E|\geq \frac{\operatorname{diam}(E)}{2}.$
By convexity, we also have 
\[
\operatorname{conv}\left(B(x_E,r_E)\cup \{x_0\}\right)\subset \overline E
\]
where $\operatorname{conv}\big(B(x_E,r_E)\cup\{x_0\}\big)$ is the convex hull of $B(x_E,r_E)\cup\{x_0\}$. Hence,
\[
 \frac{\operatorname{diam} (E)\,\omega_{n-1}\,r_E^{n-1}}{2n}\,\leq \,\frac{1}{n}|x-x_E|\,\omega_{n-1}\,r_E^{n-1}\leq\left|\operatorname{conv}\left(B(x_E,r_E)\cup \{x_0\}\right)\right|\leq |E|=\omega_n
\]
which implies \eqref{diametro inradius}
}{, with $c_n=\frac{\omega_{n-1}}{2n}$.} \end{proof}

\section{Proof of the main Theorem}
In this section, we prove Theorem \ref{thm:main}, following the strategy of Fuglede \cite{Fuglede1993}.\begin{proof}[Proof of Theorem \ref{thm:main}]
Since the quantities $\lambda_0(E)$ and $\delta_{s}(E)$ are scale invariant, we can assume without loss of generality that $|E|=\omega_n$. Up to translation, we may also assume that $\bari(E)=0$. Since $\lambda_0(E)\leq 2$, the inequality \eqref{our-quant-iso} immediately follows  for sets $E$ such that $\delta_{s}(E)\geq1$, by choosing $C\geq 2$.

Hence, from now on, let us consider a convex set $E$ with volume $\omega_n$ and barycenter at the origin,  such that $\delta_{s}(E)< 1$. Since $E$ is convex and bounded (as it has finite measure), we can parametrize its boundary as
\[
\partial E=\left\{(1+u(x))x\:|\:x\in\partial B\:\right\}
\]
for some Lipschitz function $u : \partial B\to (0,\infty)$.

Let $d(E)$ be the Hausdorff distance between $E$ and $B$, i.e.

$$ d(E):=\inf\{\,\tau\geq0:\, B_{(1-\tau)_+}\subseteq E\subseteq B_{1+\tau}\,\}$$
where $(1-\tau)_+:=\max\{1-\tau,0\}$.

We will divide the proof into two steps.\par
\medskip

\noindent \emph{Step 1.} First, let us assume that $d(E)\leq a$, for some $a$ to be chosen later. In this case we have
\begin{equation*}
\begin{split}
    \omega_n\lambda_0(E)&={|E\setminus B|+|B\setminus E|}\\&=\int_{\{u\geq0\}}(\left(1+u(x)\right)^n-1)d\mathcal{H}^{n-1}+\int_{\{u<0\}}(1-\left(1+u(x)\right)^n)d\mathcal{H}^{n-1}\\&=\int_{\partial B}|(1+u(x))^n-1|d\mathcal{H}^{n-1}\\&\leq\int_{\partial B} \sum_{j=1}^n\binom{n}{j}|u(x)|^j d\mathcal{H}^{n-1}\\&\leq\sum_{j=1}^n\binom{n}{j}|a|^{j-1}\int_{\partial B} |u(x)|d \mathcal{H}^{n-1},
\end{split}
\end{equation*}
where in the last inequality we used that $|u|\leq a$ since $d(E)\leq a$. From this last estimate we obtain \begin{equation}
    \label{eqn:baricenter estimate}\lambda_0(E)\leq\frac{(1+a)^n-1}{\omega_na}{\|u\|_{L^1(\partial B)}}.
\end{equation}

Moreover, by \cite[Lemma 2.2]{Fuglede1989} we know that  $\|\nabla u\|_{L^\infty(\partial B)}=O(\sqrt a).$ Consequently, up to choosing $a$ sufficiently small,  we can assume that $E$ is \emph{nearly spherical} with $\| u\|_{W^{1,\infty}(\partial B)}<\varepsilon_0$ where $\varepsilon_0$ is the constant appearing in Theorem \ref{fuglede}.
 Hence, by equation \eqref{fug0}, it follows
 
\begin{equation}\label{usofug}
\begin{split}
P_{s}(E)-P_{s}(B)&\geq c_0\left(\iint_{\partial B \times \partial B} \frac{|u(x)-u(y)|^2}{|x-y|^{n+s}} d \mathcal{H}_x^{n-1} d \mathcal{H}_y^{n-1}+s P_{s}(B)\|u\|_{L^2(\partial B)}^2\right)\\&\geq{c_0sP_{s}(B)}\|u\|^2_{L^2(\partial B)}\\&\geq\frac{c_0sP_{s}(B)}{n\omega_n}\|u\|^2_{L^1(\partial B)};\end{split}
\end{equation}
where $c_0$ depends only on $n$, and in the last inequality we used H\"older's inequality. Combining \eqref{usofug} with \eqref{eqn:baricenter estimate}, we find \begin{equation}\label{eqn:funzione}
\delta_{s}(E)\geq \frac{\omega_n a^2c_0s}{n((1+a)^n-1)^2}\lambda_0(E)^2,
\end{equation}
for some universal constant $a>0$ sufficiently small.\par
\medskip
\noindent \emph{Step 2.} Since $\delta_{s}(E)<1$, by Proposition \ref{prop:deficit_piccolo_implica_cubo}, there exists $l=l(n,s)$ such that $E\subseteq Q_l$. Moreover, by \cite[Pages 45--46]{Fuglede1993}, we know that  $d(E)=O(\lambda(E)^{\frac{2}{n+1}})$. Therefore, thanks to Lemma \ref{lm:decifit picclo baricentro piccolo}, we can choose $\eta_{s}=\eta(l,s,n,a)>0$ such that, if $\delta_{s}(E)\leq\eta_{s}$, then $d(E)\leq a$ where $a$ is the constant from \emph{Step 1}. Thus, by \emph{Step 1}, we infer that  \begin{equation}\label{deficit piccolo}
\delta_{s}(E)\geq\frac{ \omega_n a^2 c_0s}{n\left((1+a)^n-1\right)^2}   \lambda_0(E)^2\qquad\text{{if\quad $\delta_{s}(E)\leq \eta_s$.}}
\end{equation}{On the other hand, if $\delta_{s}(E)\ge\eta_s$ we have, recalling that $\lambda_0(E)\le2$, \begin{equation}\label{deficit grande}
\delta_{s}(E)\geq\frac{\eta}{4}\lambda_0(E)^2\qquad\text{if\quad $\delta_{s}(E)\geq\eta_s.$}\end{equation}}
which concludes the proof by setting \begin{equation}\label{costante dipende da s}{ C=\max   \left\{\frac{n^{1/2}((1+a)^n-1)}{\omega_n^{1/2}ac_0^{1/2}s^{1/2}},\frac{2}{\sqrt{\eta_{s}}}\right\}.}\end{equation}
\end{proof}
\begin{remark}\label{rmk:ns}
 It is clear that the proof of Theorem \ref{thm:main} works, not only for convex bodies, but also for nearly spherical sets $E$, as in Definition \ref{def:nearlyspherical}, such that  $\|u\|_{W^{1,\infty}(\partial B)}<\varepsilon_0$, where $\varepsilon_0\in(0,\frac{1}{2})$ is  the constant {of Theorem} \ref{fuglede}. Indeed, by retracing \emph{Step 1} of the proof of Theorem \ref{thm:main} and recalling that $d(E)<\varepsilon_0$, we have
  \begin{equation}\label{ineq nearly spherical}
  \delta_{s}(E)\geq \frac{\omega_n\varepsilon_0^2c_0 s}{n\left(\left(1+\varepsilon_0\right)^n-1\right)^2}\lambda_0(E)^2,
   \end{equation}for any $s\in(0,1).$
\end{remark}
\begin{remark} Let us fix $s\in(0,1)$ and $t\in(s,1).$ By \cite[Theorem 1.1]{DNRV} there exists a constant $D:=D(n,s,t)$, which is bounded as $t\nearrow 1,$ such that \begin{equation}\label{da s_0 a s}
\delta_t(E)\geq D\: \delta_{s}(E).
\end{equation}Combining this estimate with the result of Theorem \ref{thm:main} we find\begin{equation}\label{da s_0 a s bis}
 \lambda_0(E)\leq \frac{ C}{\sqrt D}\sqrt{\delta _t(E)},
\end{equation}for any $t\in(s,1).$
    Moreover, taking into account \eqref{convergenza}, we have that $\delta_t(E)\to\delta(E)$ as $t\nearrow1$. Hence, taking the limit as $t\nearrow 1$ in \eqref{da s_0 a s bis}, we recover
\begin{equation}\label{local}
 \lambda_0(E)\leq \frac{\gamma(n,s)}{\sqrt{D^*(n,s)}}\sqrt{\delta(E)}.
 \end{equation}
for any convex set $E$ with finite measure and any $s\in(0,1)$. Here $D^*(n,s)$ is defined as $$D^*(n,s)=\limsup_{t\nearrow 1}D(n,s,t)$$ and the constant ${\gamma(n,s)}$ is given by \[
 {\gamma(n,s)}=\max   \left\{\frac{n^{1/2}((1+a)^n-1)}{\omega_n^{1/2}ac_0^{1/2}s^{1/2}},\frac{2}{\sqrt{\eta_{s}}}\right\}
, \] 
 with $a>0$ universal (depending only on $n$) chosen as in \emph{Step 1} of the proof of Theorem \ref{thm:main} and $\eta_{s}$ defined as in Lemma \ref{lm:decifit picclo baricentro piccolo}.
  Note that \eqref{local} is consistent with the results obtained in \cite{Fuglede1993}, \cite{BCH} and \cite{GP}.
\end{remark}
\begin{ack}
 The first author wishes to thank the Department of Mathematics of the University of Bologna for organizing the ASK Conference in December 2024, during which this work was initiated. 
   The authors are members of the Gruppo Nazionale per l’Analisi Matematica, la Probabilità e le loro
   Applicazioni (GNAMPA) of the Istituto Nazionale di Alta Matematica (INdAM). E.M.M., B.R. and M.T. were partially funded by the INdAM--GNAMPA Project "Ottimizzazione Spettrale, Geometrica e Funzionale", CUP: E5324001950001, and by the PRIN project 2022R537CS "NO$^3$--Nodal Optimization, NOnlinear elliptic equations, NOnlocal geometric problems, with a focus on regularity", CUP: J53D23003850006.
\end{ack}

\subsection*{Declarations}
The authors declare that there is no conflict of interest.
\bibliographystyle{abbrv}
\bibliography{biblio}
\end{document}